\documentclass[11pt]{article}

\usepackage{amsmath, amssymb, amsthm, enumerate}
\usepackage{graphicx, caption, subcaption}
\usepackage{xcolor}
\usepackage[pagebackref]{hyperref}  %

\usepackage{titlesec}
\theoremstyle{plain}
\newtheorem{thm}{Theorem}[section]
\newtheorem{prop}[thm]{Proposition}
\newtheorem{lemma}[thm]{Lemma}

\newtheorem{claim}[thm]{Claim}

\theoremstyle{definition}

\newtheorem{question}[thm]{Question}
\newtheorem{rmk}[thm]{Remark}

\newcommand{\E}{\mathbb{E}}

\renewcommand{\P}{\mathbb{P}}

\newcommand{\tr}{\operatorname{Trace}}

\newcounter{bencomments}

\newcounter{jenyacomments}

\title{Simple vs non-simple loops on random regular graphs }  
\author{Benjamin Dozier \thanks{Department of Mathematics, Cornell University, \href{mailto: benjamin.dozier@cornell.edu}{\nolinkurl{benjamin.dozier@cornell.edu}}} \and  Jenya Sapir \thanks{Department of Mathematics, Binghamton University, \href{mailto:sapir@math.binghamton.edu}{\nolinkurl{sapir@math.binghamton.edu}}.} }

\begin{document}
\maketitle

\begin{abstract}
  In this note we solve the ``birthday problem'' for loops on random regular graphs.  Namely, for fixed $d\ge 3$, we prove that on a random $d$-regular graph with $n$ vertices, as $n$ approaches infinity, with high probability: 
  \begin{enumerate}[(i)]
  \item almost all primitive non-backtracking loops of length $k \prec \sqrt{n}$ are simple, i.e. do not self-intersect, 
  \item almost all primitive non-backtracking loops of length $k \succ \sqrt{n}$ self-intersect.  
  \end{enumerate}
\end{abstract}

\section{Introduction}
\label{sec:intro}

Observe that the shortest non-backtracking loop on any regular graph
is simple i.e. passes through each vertex at most once.  As we consider non-backtracking loops of length $k$ getting larger, eventually all of them are non-simple, since there are no simple loops of length greater than $n$, the number of vertices.   At what length $k$ (as a function of $n$) do the non-simple loops start to become more common?  The scale at which the transition happens will depend on the shape of the graph; in this paper we answer the question for the case of \emph{random} regular graphs of fixed degree $d\ge 3$.  Our result is that the transition occurs around the threshold $k=\sqrt{n}$. 

Our interest in this question arose from consideration of the analogous question for hyperbolic surfaces, raised by Lipnowski-Wright in \cite[Conjecture 1.2]{LW}.  

\bigskip 

\paragraph{Terminology.}

We define a \emph{walk} of length $k$ on a (multi)graph to be a sequence of oriented edges $e_1e_2\cdots e_k$ with the terminal vertex of $e_i$ equal to the initial vertex of $e_{i+1}$ for $i=1,\ldots,k-1$.  We say a walk is a \emph{loop} if the terminal vertex of $e_k$ is equal to the initial vertex of $e_1$.  If the edge sequences of two loops differ by a cyclic shift, we consider the two loops to be the same.   We say a walk is \emph{non-backtracking} if $\overline e_i\ne e_{i+1}$ (here $\overline e_i$ denotes $e_i$ with orientation reversed) for $i=1,\ldots,k-1$.   We say a loop is non-backtracking if \emph{all} of the walks associated to that loop are non-backtracking (in particular, when a distinguished start vertex is chosen, we must have  ``non-backtracking at the start'').   A walk is said to be \emph{simple} if the terminal endpoints of the edges are distinct, and a loop is said to be simple if all (equivalently, any) of the associated walks are simple.  We say a loop is \emph{primitive} if it is not the repetition of a shorter loop. 

\subsection{Model of random regular graphs and notation}

Let $G(d,n)$ be a random $d$-regular graph on $n$ vertices, chosen using the \emph{configuration model}.   One samples from this distribution as follows.  Begin with $n$ vertices, each with $d$ unpaired half-edges.  At each stage choose some unpaired half-edge and pair it with a different unpaired half-edge chosen uniformly at random.  Repeat until all half-edges are paired.  The resulting ``graph'' can have self-loops and multiple edges between a pair of vertices; nevertheless we will abuse terminology and call them graphs.  

Given $\Gamma$ a $d$-regular graph let:

\begin{itemize}
\item $N_{simp}(\Gamma,k)$ be the number of \emph{simple}  non-backtracking loops on $\Gamma$ of length $k$,
\item $N_{prim}(\Gamma,k)$ be the number of \emph{primitive} non-backtracking loops on $\Gamma$ of length $k$.   
  \end{itemize}
Let $\P_n[\cdot]$ denote the probability of an event with respect to graphs drawn from $G(d,n)$, and $\E_n[\cdot]$ denote the expected value.  (Note that these quantities depend also on the degree $d$, but we will always think of $d$ as fixed).  We will use the shorthand $\E_n[X(k)] := \E_n[X(\Gamma,k)]$ for $X$ any of the quantities in the list above.

\subsection{Main results}
\label{sec:main-results}

\begin{thm}[Low length regime]
  \label{thm:graph-short}
  Take $d\ge 3$ fixed.  Suppose $k$ is some function of $n$ satisfying $k\prec \sqrt{n}$.  Fix $\epsilon>0$. Then
  $$\P_n\left[ N_{simp}(k)\ge (1-\epsilon) N_{prim}(k)\right] \to 1, $$
  as $n\to\infty$.
\end{thm}

\begin{thm}[High length regime]
  \label{thm:graph-long}
  Take $d\ge 3$ fixed.  Suppose $k$ is some function of $n$ satisfying $k\succ \sqrt{n}$.  Fix $\epsilon>0$. Then
  $$\P_n\left[ N_{simp}(k) \le \epsilon \cdot N_{prim}(k)\right] \to 1, $$
  as $n\to\infty$. 
\end{thm}

In this paper, $A\prec B$ means that $A=o(B)$, and $A \sim B$ means $\lim \frac AB = 1$.  

In both theorems above, when $d$ is odd, for there to be any $d$-regular graphs on $n$ vertices, $n$ must be even.  Thus in those cases we take $n\to\infty$ along the even integers.  
\begin{rmk}
  The above two theorems also hold if we replace the configuration model $G(d,n)$ with the \emph{uniform model} over all $d$-regular graphs, without self-loops or multiple edges, on $n$ vertices.  These versions can be deduced from the theorems above together with the result that the probability that a (multi)graph from $G(d,n)$ has neither self-loops nor multiple edges tends to a positive constant as $n\to\infty$  (see \cite{Bollobas80} or \cite[Theorem 2.16]{Bollobas01}).
\end{rmk}

\begin{rmk}
  Theorem \ref{thm:graph-short} would \emph{not} be true if we replaced $N_{prim}$ by $N_{all}$, the number of non-backtracking loops, without the primitive condition.  In particular, it would fail for $k$ a fixed composite integer $pq$.  In fact, in that case, it is known  (by \cite{Bollobas80}) that $N_{simp}(k)$ and $N_{simp}(p)$ asymptotically have finite positive mean, and since $\E_n[N_{all}(k)] \ge \E_n[N_{simp}(k)] + \E_n[N_{simp}(p)]$, we get that $N_{all}$ must be greater than $N_{simp}$ by a definite factor a positive proportion of the time.  
\end{rmk}

\paragraph{Non-random graphs.}

Note that the analog of Theorem \ref{thm:graph-short} for fixed sequences of regular graphs (rather than random ones) fails.  Families of graphs with diameter linear in $n$ provide counter-examples, since in this case walks behave like random walk on a line and thus even short walks (and loops) are likely to self-intersect.

On the other hand, we do not know if the non-random version of Theorem \ref{thm:graph-long} holds: 
\begin{question}
  Let $n_j$ be an increasing sequence of positive integers, $\Gamma_{n_j}$ a $d$-regular graph on $n_j$ vertices, and $k$ some function of $n$ with $k(n)\succ\sqrt{n}$.  Is it true that for any fixed $\epsilon>0$,
  \begin{align*}
    N_{simp}(\Gamma_{n_j}, k(n_j)) < \epsilon \cdot N_{prim}(\Gamma_{n_j},k(n_j))
  \end{align*}
  for all $j$ sufficiently large (depending on $\epsilon$)? 
\end{question}

\subsection{Discussion of the proofs}
\label{sec:discussion-proofs}

\paragraph{Heuristic:} A random loop of length $k$ on a random graph should behave in some sense like choosing a list of $k$ vertices uniformly at random from all $n$ vertices and making the loop travel through them in order.  If this were the case, then the solution to the standard ``birthday problem'' for picking $k$ birthdays randomly from among $n$ suggests that the transition between all the vertices being distinct versus having at least one repetition should occur around $k=\sqrt{n}$.

\bigskip

An immediate issue with the above heuristic is that on a $d$-regular graph, a non-backtracking walk can only be extended by one step in $d-1$ ways, so certainly not all $n$ vertices are equally likely to come after some given vertex.  However, since we are choosing the graph randomly as well, one can in fact think of the next vertex along a walk as being randomly selected from the $n$ vertices.  This is because we can build the graph and walk simultaneously, only making choices about the graph when the walk forces us to.  There are, however, two issues with doing this:
\begin{enumerate}[(i)] 
\item  it only gives control over expected values of counts; to get asymptotic almost sure (a.a.s.\@) control requires additional work,
\item it fails if the walk already self-intersects, since in that case we would not get the necessary freedom in the choice of next edge to follow.
\end{enumerate}

If our ultimate goal was to study \emph{walks} rather than loops, neither of these issues would be problematic.  In fact the technique suggested above easily gives that the transition for a random non-backtracking walk on a random graph to be self-intersecting occurs around $k=\sqrt{n}$.  Conditioning on the walk being a loop makes things considerably harder.  Unlike the number of non-backtracking walks, the number of non-backtracking loops depends on the particular graph.  This means that knowing the expected number of simple loops is not enough; we must know additional information about the distribution of the number of simple loops, and separate information about the count of primitive loops.

\subsection{Outline of paper}
\label{sec:outline}
\begin{itemize}
\item In Section \ref{sec:count-simp-loop} we compute the expected number of simple loops.   In the low length regime $k\prec \sqrt{n}$, we also bound the second moment of the number of simple loops and then use this to control the a.a.s.\@ behavior.  

\item In Section \ref{sec:count-prim-loop} we study the count of primitive loops.  In the \emph{very} low length regime $k\prec n^{1/4}$, we control the expected value, while for $k\succ \log n$ we control the a.a.s.\@ behavior.
  
\item In Section \ref{sec:simple-prim} we combine the results from the previous sections to prove both the main theorems.  
\end{itemize}

\paragraph{Acknowledgments:}  We would like to thank Noga Alon, Lionel
Levine, and Alex Wright for useful conversations.

\section{Counting simple loops}
\label{sec:count-simp-loop}

We will see below that estimating \emph{expected} counts of simple loops is relatively easy (Proposition \ref{prop:e-simp-short} and Proposition \ref{prop:e-simp-long}), using the idea of building the graph and walk simultaneously.  However, we will also need control of the a.a.s.\@ behavior, so we need to rule out high variance.  In the low length regime $1\prec k\prec \sqrt{n}$, we estimate the second moment (Proposition \ref{prop:var-simp-short}), and then use the second moment method to deduce the a.a.s.\@ behavior (Proposition \ref{prop:aas-simple-short}).  In the high length regime $k\succ \sqrt{n}$, since we are trying to prove an inequality of the form $N_{simp}\le \epsilon N_{prim}$, the first moment method will suffice.

\subsection{Expected count of simple loops}
\label{sec:e-simple}

\begin{prop}
  \label{prop:e-simp-short}
  Let $k\prec \sqrt{n}$.  Then 
  \begin{align*}
    \E_n[N_{simp}(k)] \sim \frac{ (d-1)^k}{k},
  \end{align*}
  as $n\to\infty$.
\end{prop}

\begin{proof}
  The main idea is to build the graph and walk simultaneously.  A similar result, with a similar proof, appears in \cite[Lemma 4]{BS}.
  Let $p$ denote the probability that a randomly chosen non-backtracking walk on a random regular graph is closed and simple.  To compute $p$, we will consider choosing the random graph at the same time as the random walk, only making choices about which half-edges are paired in the graph when we are forced to.  The first factor below is the probability that the first half-edge the path follows is paired with a half-edge that leads to a different vertex (since there are $dn-1$ half-edges it could be paired with, of which $d-1$ are incident to the start vertex).  Continuing in this way, we see that
  \begin{align*}
    p= \left(1-\frac{d-1}{dn-1}\right) \left(1-\frac{(d-1)+(d-2)}{dn-3}\right) \cdots \left(1- \frac{(d-1)+(k-2)(d-2)}{dn-(2k-3)}\right)\left(\frac{d-1}{dn-(2k-1)}\right).
  \end{align*}
  
  Since, a fortiori, $k\prec n$, we get
  \begin{align*}
    p \sim \left(1-\frac{1}{n}\right) \left(1-\frac{2}{n}\right) \cdots \left(1-\frac{k-1}{n}\right) \left( \frac{d-1}{d} \cdot \frac{1}{n}\right).
  \end{align*}

  Then using the estimate $1-x = \exp(-x+O(x^2))$ for small $x$ repeatedly gives
  \begin{align*}
    p & \sim \exp(-k^2/n) \left( \frac{d-1}{d} \cdot \frac{1}{n}\right)   \\
      & \sim  \frac{d-1}{d} \cdot \frac{1}{n},
  \end{align*}
  where in the last step we have used the assumption that $k\prec \sqrt{n}$.

  Now to compute the desired expected value, we note that every regular graph has $nd(d-1)^{k-1}$ non-backtracking walks of length $k$.  Each simple loop will be counted $k$ times in this way.  Hence
  \begin{align*}
    \E_n [k \cdot N_{simp}(k)] &= nd(d-1)^{k-1}p \\
    & \sim nd(d-1)^{k-1}\left( \frac{d-1}{d} \cdot \frac{1}{n}\right) \\
    &\sim (d-1)^k,
  \end{align*}
which then gives the desired result.  
\end{proof}

\begin{prop}
  \label{prop:e-simp-long}
  Suppose $k$ is some function of $n$ satisfying $k\succ \sqrt{n}$.  Then
  $$\E_n\left[ N_{simp}(k)\right] \prec \frac{(d-1)^k}{k},$$
  as $n\to\infty$.  
\end{prop}

\begin{proof}
  Let $p$ denote the probability that a randomly chosen non-backtracking walk on a random regular graph is a closed and simple.  We recall the exact formula for $p$ used in the proof of Proposition \ref{prop:e-simp-short}: 
  \begin{align*}
    p= \left(1-\frac{d-1}{dn-1}\right) \left(1-\frac{(d-1)+(d-2)}{dn-3}\right) \cdots \left(1- \frac{(d-1)+(k-2)(d-2)}{dn-(2k-3)}\right)\left(\frac{d-1}{dn-(2k-1)}\right).
  \end{align*}

  Using that the denominators in the above are increasing and numerators are decreasing from one fraction to the next (excluding the last term, for which we use a different estimate), we see that
  \begin{align*}
    p\le \left(1- \frac{(k/2)(d-2)}{dn-(2k-3)} \right)^{(k-1)/2} \cdot O(1/n).  
  \end{align*}
  (Note that we can assume $k\le n$, since otherwise there are no simple walks of length $k$).  Using the approximation $1-x = \exp(-x+O(x^2))$, and the assumption that $k\succ \sqrt{n}$, we get from the above that
  \begin{align*}
    p=o(1) \cdot \frac{1}{n}.  
  \end{align*}

  Then, as in the proof of Proposition \ref{prop:e-simp-short}, we get
  \begin{align*}
    \E_n[k\cdot N_{simp}(k)] &= nd(d-1)^kp \\
                             &=nd(d-1)^k\cdot o(1) \cdot \frac{1}{n}.   \\
                             &=o\left((d-1)^k\right),
  \end{align*}
which implies the desired result.
\end{proof}

\subsection{Variance of count of simple loops}
\label{sec:var-simple}

\begin{prop}
  \label{prop:var-simp-short}
  Let $1 \prec k\prec \sqrt{n}$.  Then 
  \begin{align*}
    \E_n[N^2_{simp}(k)] \le (1+o(1)) \cdot \frac{ (d-1)^{2k}}{k^2},
  \end{align*}
  as $n\to\infty$.
\end{prop}

\begin{proof}
  Given a graph $\Gamma$, let $L_{dist}$ denote the number of ordered pairs $(\gamma_1,\gamma_2)$, where each $\gamma_i$ is an oriented simple loop on the graph $\Gamma$ with a choice of distinguished point on the loop, and such that $\gamma_1\ne \gamma_2$ and $\gamma_1 \ne \overline \gamma_2$ (here $\overline \gamma$ denotes the loop $\gamma$ but with orientation reversed).

  Since there are exactly $nd(d-1)^{k-1}$ non-backtracking paths of length $k$, we have
  \begin{align*}
    \E_n[L_{dist}] = \left(nd(d-1)^{k-1}\right)^2 p,
  \end{align*}
where $p$ is the probability for $\Gamma$ a randomly chosen regular graph and $\gamma_1$, $\gamma_2$ randomly chosen non-backtracking walks on $\Gamma$, that $\gamma_1,\gamma_2$ are both simple loops with $\gamma_1\ne \gamma_2$, $\gamma_1 \ne \overline \gamma_2$.  We can write $p=p_1 + p_2$, where
\begin{itemize}
\item $p_1$ is the probability (under the same choices as above) that $\gamma_1,\gamma_2$ are both simple loops and that the initial vertex of $\gamma_1$ is \emph{not} one of the vertices of $\gamma_2$,
\item $p_2$ is the probability that $\gamma_1,\gamma_2$ are both simple loops, the initial vertex of $\gamma_1$ \emph{is} one of the vertices of $\gamma_2$, and $\gamma_1\ne \gamma_2$, $\gamma_1 \ne \overline \gamma_2$.
\end{itemize}

To bound these probabilities from above, we will again consider choosing the random graph at the same time as the random walk, only making choices about the graph when we are forced to.

To bound $p_1$, note that the probability that the last edge of $\gamma_1$ goes back to its initial vertex $v_1$ is at most $(1+o(1))\frac{d-1}{d}\frac 1n$, where the $o(1)$ bound on the error uses that $k \prec n$. The analogous statement is true for $\gamma_2$ since the initial vertex of $\gamma_2$ is disjoint from $\gamma_1$.  Hence we get
\begin{align*}
  p_1 \le (1+o(1))\left(\frac{d-1}{d} \cdot \frac{1}{n}\right)^2.  
\end{align*}

\begin{figure}[h!]
 \centering 
 \includegraphics{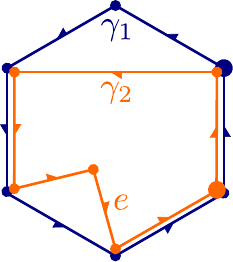}
 \caption{When the initial vertex of $\gamma_2$ lies on $\gamma_1$.}
 \label{fig:picture}
\end{figure}

For $p_2$, we get a factor of $O(1/n)$ from the condition that $\gamma_1$ is a simple loop.  We also get factor of $O(k/n)$ from the condition that the initial vertex of $\gamma_2$ coincides with a vertex of $\gamma_1$.  Now if $(\gamma_1,\gamma_2)$ is a pair of simple loops in $L_{dist}$, there exists a unique edge $e$ of $\gamma_2$ such that $e$ is part of neither $\gamma_1$ nor $\bar \gamma_1$, and such that after $\gamma_2$ traverses $e$, it exactly follows either $\gamma_1$ or $\bar \gamma_1$ until its final vertex (Figure \ref{fig:picture}).  There are $k$ choices of where along $\gamma_2$ this edge $e$ is, and the probability that the forward endpoint of $e$ coincides the appropriate vertex of $\gamma_1$ is $O(1/n)$.  Thus we get a further factor of $O(k/n)$.  Putting this all together, we find
\begin{align*}
  p_2 &  =  O\left( \frac{1}{n} \cdot \frac{k}{n} \cdot \frac{k}{n}\right) \\
  &=o\left(\frac{1}{n^2}\right), 
\end{align*}
where in the last step we have used the assumption that $k^2 \prec n$.

Putting the two estimates together gives
\begin{align*}
  p = p_1 + p_2 &\le  (1+o(1))\left(\frac{d-1}{d} \cdot \frac{1}{n}\right)^2 + o\left(\frac{1}{n^2}\right) \\
  &\le  (1+o(1))\left(\frac{d-1}{d} \cdot \frac{1}{n}\right)^2.  
\end{align*}

Now returning to $L_{dist}$, we get
\begin{align*}
  \E_n[L_{dist}] &= \left(nd(d-1)^{k-1}\right)^2 p \\
  & \le \left(nd(d-1)^{k-1}\right)^2 (1+o(1))\left(\frac{d-1}{d} \cdot \frac{1}{n}\right)^2 \\
  &\le (1+o(1)) \cdot (d-1)^{2k}.
\end{align*}

  To finish the proof we combine the above with the term coming from pairs of loops that are either identical or differ only in orientation. There are $k^2N_{simp}^2$ pairs of loops with distinguished start points. So using Proposition \ref{prop:e-simp-short} we compute:
  \begin{align*}
    \E_n[k^2\cdot N_{simp}^2] &= \E_n[L_{dist} +  2k^2 \cdot N_{simp}] = \E_n[L_{dist}] + \E_n[2k^2 \cdot N_{simp}] \\
                         &\le  (1+o(1)) \cdot (d-1)^{2k} + 2(1+o(1)) \cdot k(d-1)^k\\
                         &\le (1+o(1)) \cdot (d-1)^{2k},
  \end{align*}
  where in the last line we have used the assumption that $1\prec k$.  Dividing by $k^2$ gives the desired result.  
  
\end{proof}

\subsection{Asymptotic almost sure behavior of simple loops}
\label{sec:aas-simple}

\begin{prop}
  \label{prop:aas-simple-short} 
Let $1\prec k\prec \sqrt{n}$.  Fix $\epsilon>0$.  Then
$$\lim_{n\to\infty} \P_n\left[ \left|\frac{N_{simp}(k)}{(d-1)^k/k} -1 \right| < \epsilon \right] = 1.$$
\end{prop}

\begin{rmk}
  Note that above is false for $k$ a constant, since in that case $N_{simp}$ has a Poisson distribution with positive variance (\cite{Bollobas80}, Theorem 2). %
\end{rmk}
\begin{proof}
We have computed the first and second moments of $N_{simp}$ above, so to obtain the a.a.s.\@ behavior we use the second moment method.  For any $\delta>0$, we have, using Proposition \ref{prop:var-simp-short}
  \begin{align*}
    \P_n[|N_{simp}-\E_nN_{simp}|\ge \delta] &= \P_n[(N_{simp}-\E_nN_{simp})^2\ge \delta^2 ] \\
                                          &\le \frac{\E_n[(N_{simp}-\E_nN_{simp})^2]}{\delta^2} \\
                                          &= \frac{\E_n[N_{simp}^2]-\E_n[N_{simp}]^2}{\delta^2}\\
                                          &\le \frac{(1+o(1)) (d-1)^{2k}/k^2 - (d-1)^{2k}/k^2}{\delta^2} \\
                                          &\le \frac{o(1) (d-1)^{2k}/k^2}{\delta^2}.
  \end{align*}
Taking $\delta=\epsilon(d-1)^k/k$ and dividing by $(d-1)^k/k$ gives 
\begin{align*}
  \P_n\left[\left|\frac{N_{simp}}{(d-1)^k/k}-\frac{\E_n N_{simp}}{(d-1)^k/k}\right| \ge \epsilon \right] \le \frac{o(1)}{\epsilon^2}.  
\end{align*}

Hence for any $\epsilon>0$
\begin{align*}
  \lim_{n\to\infty} \P_n\left[\left|\frac{N_{simp}}{(d-1)^k/k}-\frac{\E_n N_{simp}}{(d-1)^k/k}\right| \ge \epsilon \right] = 0, 
\end{align*}
 and then using Proposition \ref{prop:e-simp-short}, which states that $\lim_{n\to \infty} \frac{\E_n N_{simp}}{(d-1)^k/k}=1$, we get
\begin{align*}
  \lim_{n\to\infty} \P_n\left[ \left|\frac{N_{simp}(k)}{(d-1)^k/k} -1 \right| \ge \epsilon \right]  & =0,
\end{align*}
as desired.  

\end{proof}

\section{Counting primitive loops}
\label{sec:count-prim-loop}

When the length is very low ($k\prec n^{1/4}$), we will show that the probability that a non-backtracking walk forms \emph{two} loops in its induced subgraph is so low that this probability is still negligible conditioned on the event that the walk is a loop.  It follows that in this regime the expected number of \emph{primitive} loops (Proposition \ref{prop:e-prim-short}) has the same asymptotics as the expected number of simple loops, computed in the previous section.  

When $k\succ \log n$, we use the spectral gap for the adjacency matrix $A$ of a random graph to control the a.a.s.\@ count of loops (Proposition \ref{prop:aas-prim-long}) The trace of $A^k$ counts all loops of length $k$, without the non-backtracking condition.
Since we are interested in non-backtracking loops, we study the related ``non-backtracking matrix'' $\tilde A$ whose eigenvalues can be computed in terms of those of $A$.

\subsection{Length $k\prec n^{1/4}$}

\begin{prop}
  \label{prop:e-prim-short}
Let $d\ge 3$, and let $k\prec n^{1/4}$.  Then 
  \begin{align*}
    \E_n[N_{prim}(k)] \sim \frac{ (d-1)^k}{k}
  \end{align*}
  as $n\to\infty$.
\end{prop}

\begin{proof}

  We will compute the expected value as follows.  On any graph from $G(d,n)$, the total number of non-backtracking walks $\gamma$ of length $k$ equals
  \begin{align}
    \label{eq:nonback}
    nd(d-1)^{k-1},
  \end{align}
  since there are $n$ choices of starting vertex, then $d$ choices for the first outgoing edge, and $(d-1)$ choices for the succeeding outgoing edges (note that if the walk happens to be a loop, then the resulting loop could potentially backtrack at the starting vertex).

  We now consider choosing such a $\gamma$ randomly, i.e. we choose a graph randomly according to $G(d,n)$, pick a random start vertex, and then pick a random non-backtracking walk starting at that vertex.  Combined with \eqref{eq:nonback}, to prove the Proposition, it will suffice to compute the probability $p_{prim}$ that $\gamma$ is a loop that is primitive and non-backtracking.  
  
  We begin by showing that $\gamma$ forming multiple loops is unlikely.

  \medskip
  \begin{claim}
The probability that the induced graph formed by $\gamma$ has at least two loops is $O(k^4/n^2)$.
\end{claim}

    \begin{proof}
      This appears as \cite[Lemma 3]{BS}.  We proceed by building the random graph and the random walk at the same time.  If there are two loops it means that we had at least two ``free choices'' of an edge that came back to vertices already on the walk.  By free, we mean that we picked a half-edge to follow that was unpaired.  The probability that the half-edge that this gets paired to is incident to one of the vertices already in the walk is $O(k/(n-k))$ which, since we are assuming $k \prec n^{1/4}$, is $O(k/n)$.  There are $\binom{k}{2}$ choices for the two steps at which the collision occurs.  Thus the probability of interest is $O(k^2(k/n)^2)$, as desired.
    \end{proof}
          
    \begin{claim}
      Given a non-backtracking, primitive, non-simple loop $\gamma$ on a graph, its induced subgraph must contain at least two loops.
    \end{claim}
    \begin{proof}
      Let $k$ be the length of the loop.  Choose a starting point $v_1$, and let $\gamma$ traverse the vertices $v_1,\ldots,v_k,v_{k+1}
      =v_1$ in that order.  We can choose the starting point such that a simple loop $\alpha$ is formed by $v_1,v_2,\ldots,v_i=v_1$  for some $1<i\le k$ (we use both non-simple and non-backtracking properties here).  
      After $v_i$, the walk may follow this loop $\alpha$ several times, but since the loop $\gamma$ is primitive, eventually it must depart from $\alpha$, say at step $j$.  A new loop $\beta$ is then formed somewhere between steps $j$ and $k$.
    \end{proof}
    
    Now we will compute the probability that $\gamma$ is a simple loop.   By Proposition \ref{prop:e-simp-short}, the average number of simple loops tends to $(d-1)^k/k$.  This means that the average number of walks with a distinguished start vertex that form a simple loop tends to $(d-1)^k$.  Now recall that in \eqref{eq:nonback} above, we showed that on any $d$-regular graph, the number of length $k$ non-backtracking walks is exactly $nd(d-1)^{k-1}$.  These facts together mean that the probability that $\gamma$ is a simple loop is 
    \begin{align}
      p_{s}\sim \frac{(d-1)^k}{nd(d-1)^{k-1}} = \frac{1}{n} \frac{d-1}{d}.
      \label{eq:simple}
    \end{align}
    By the two Claims above, the probability that $\gamma$ is a primitive loop that is non-backtracking (including at start vertex) and \emph{non-simple} is
    \begin{align*}
            p_{ns}=O(k^4/n^2) = o(1/n)
    \end{align*}

    Combined with \eqref{eq:simple}, we get that the probability that $\gamma$ is a primitive, non-backtracking loop is
    \begin{align*}
      p_{prim}= p_s+p_{ns}\sim \frac{1}{n} \frac{d-1}{d} + o(1/n) \sim \frac{1}{n} \frac{d-1}{d}.
    \end{align*}

    We combine this with \eqref{eq:nonback} to compute the expected number of primitive, non-backtracking loops of length $k$, with a distinguished start-vertex to be 
    \begin{align*}
      nd(d-1)^{k-1} \cdot p_{prim} \sim nd(d-1)^{k-1} \cdot \frac{1}{n} \frac{d-1}{d} = (d-1)^k.
    \end{align*}

    There are $k$ choices of distinguished start point along the loop, so if we forget this, then the quantity goes down by a factor of $k$, giving the desired result. 
  \end{proof}
           
\subsection{Length $k\succ \log n$}

  \paragraph{Non-backtracking matrix.}
  
 Let $A$ be the $n\times n$ adjacency matrix of $d$-regular graph $\Gamma$.  Since $\Gamma$ is undirected, this graph is symmetric, so $A$ has $n$ real eigenvalues
  $$\lambda_1 \ge \lambda_2 \ge \cdots \ge \lambda_n.$$
  Since $\Gamma$ is $d$-regular, $\lambda_1=d$.  Let $\lambda = \max(|\lambda_2|,|\lambda_n|)$.  

  Powers of $A$ and associated spectral information can be used to understand counts of walks with backtracking allowed.  Since we are interested in non-backtracking walks, we consider the related Markov process whose state space is the set of directed edges of $A$. Two states $e_1,e_2$ are connected if $e_1,e_2$ are not opposite orientations of the same edge, and the forward endpoint of $e_1$ equals the back endpoint of $e_2$ (these conditions mean that going from $e_1$ to $e_2$ is a valid move in a non-backtracking walk).  We can think of this Markov process as a directed graph $\tilde \Gamma$ with $nd$ vertices, with each vertex having $d-1$ incoming edges and $d-1$ outgoing edges.  Observe that there is a bijection between non-backtracking loops on $\Gamma$ and all (directed) loops on $\tilde \Gamma$.

  To count non-backtracking loops, we will study the adjacency matrix $\tilde A$ of $\tilde \Gamma$.  We call $\tilde A$ the \emph{non-backtracking matrix} associated to the graph $\Gamma$.  
We denote by $\mu$ the magnitude of the second largest number in magnitude among the eigenvalues of   $\tilde A$ (listed with multiplicity).

The following lemma says that if $A$ has spectral gap, then so does $\tilde A$. 
\begin{lemma}
  \label{lem:nonback-spec-gap}
  Let $A$ be the adjacency matrix of any regular graph, and $\tilde A$ the corresponding non-backtracking matrix. For every $\epsilon > 0$, there is some $\delta > 0$ so that if 
  $$\lambda<d-\epsilon$$
for some $\epsilon>0$ then    $$\mu<(d-1)-\delta.$$

\end{lemma}

\begin{proof}
  Using Ihara's theorem, Glover and Kempton give the eigenvalues of $\tilde A$ in terms of those of $A$ as
  \begin{align*}
    \pm 1,\text{ } \mu^\pm_i:= \frac{\lambda_i \pm \sqrt{\lambda_i^2-4(d-1)}}{2}.  
  \end{align*}
  where the $\pm 1$ eigenvalues occur with multiplicity greater than 1 (and in fact $\tilde A$ is diagonalizable) \cite[Theorem 2.2]{GK}.  

  Note that $\mu^+_1=d-1$, since $\lambda_1=d$.  Let $f(x)= \frac{x + \sqrt{x^2-4(d-1)}}{2}$.

  \begin{claim}
 \begin{align*}
   \mu \le \max\left\{ \sqrt{d-1}, f(d-\epsilon) \right\}.
 \end{align*}
 (Note that if necessary, we can decrease $\epsilon$ so that $f(d-\epsilon)$ is real).  
\end{claim}

\begin{proof}
  
  If $|\lambda_i|< 2\sqrt{d-1},$ then both $\mu^\pm_i$ are non-real, and have magnitude equal to $\sqrt{d-1}$, which can be seen by multiplying by the conjugate.

  If $|\lambda_i| \ge 2\sqrt{d-1}$, and $i\ne 1$, then both $\mu^\pm_i$ are real.  In this case, if $\lambda_i\ge 0$, then
  \begin{align*}
    0\le \mu_i^- \le \mu_i^+ = f(\lambda_i) \le f(\lambda) \le f(d-\epsilon),
  \end{align*}
  where in the last two inequalities, we have used that $f$ is an increasing function on $[2\sqrt{d-1},\infty)$.
When $\lambda_i\le 0$, arguing similarly gives $ |\mu^\pm_i|\le f(d-\epsilon)$, completing the proof.  
\end{proof}

Applying this Claim we get
\begin{align*}
  \mu \le \max\left\{ \sqrt{d-1}, f(d-\epsilon) \right\} &\le  f(d) - \min\left\{f(d)-\sqrt{d-1}, \ f(d)-f(d-\epsilon)\right\} \\
  &=(d-1)- \min\left\{d-1-\sqrt{d-1}, \ f(d)-f(d-\epsilon)\right\}.
\end{align*}
Combining this with the fact that $f$ is a \emph{strictly} increasing function on $[2\sqrt{d-1},\infty)$ gives the existence of $\delta$ with the desired property.  
\end{proof}

\begin{lemma}
  \label{lem:qual-spec-gap}
    Let $\tilde A$ be the non-backtracking matrix for a graph from $G(d,n)$, and let $\mu$ be the largest magnitude among the eigenvalues of $\tilde A$ other than $\mu_1$.  Then there exists $\delta>0$ such that 
  \begin{align*}
    \lim_{n\to\infty} \P_n[\mu<(d-1)-\delta] = 1.
  \end{align*}
\end{lemma}

\begin{proof}
  
  Let $\lambda$ be the quantity for $A$ defined above.  By \cite[Theorem 4.2]{Alon}, there is some $\epsilon>0$ such that
    \begin{align*}
    \lim_{n\to\infty} \P_n[\lambda<d-\epsilon] = 1.
    \end{align*}
    Then applying Lemma \ref{lem:nonback-spec-gap} to such an $A$ gives a $\delta$ such that $\mu<(d-1)-\delta$.

\end{proof}

\begin{prop}
  \label{prop:aas-prim-long}
  Suppose $k$ is some function of $n$ satisfying $k\succ \log n$.  Fix $\epsilon>0$.  Then
  $$\P_n\left[  1-\epsilon < \frac{N_{prim}(k)}{(d-1)^k/k} < 1+\epsilon \right] \to 1,$$
  as $n\to\infty$.
\end{prop}

\begin{proof}
  The non-backtracking matrix $\tilde A$ associated to a $d$-regular graph has dimensions $nd\times nd$.  Let $m=nd$.    
  By the same result used in proof of Lemma \ref{lem:nonback-spec-gap} $\tilde A$ is diagonalizable \cite[Theorem 2.2]{GK}.  Let
  $$\mu_1=d-1,\mu_2,\ldots, \mu_m,$$
  be the eigenvalues of $\tilde A$, listed with multiplicity, arranged in order of decreasing magnitude.  

  By Lemma \ref{lem:qual-spec-gap} there exists $\delta>0$ such that
  \begin{align}
    \label{eq:spec-gap}
    |\mu_i| < (d-1)-\delta, \ \text{ for } i=2,\ldots,m,
  \end{align}
  with probability tending to $1$ as $n\to\infty$.

Now note that the $j$th diagonal entry of $\tilde A^k$ counts the number of directed walks on $\tilde \Gamma$ that start and end at the $j$th vertex.  We denote by $N_{tr}(k)$ the number of non-backtracking walks on $\Gamma$ that start and end at the same vertex.  Then

\begin{align*}
  N_{tr}(k) = \tr (\tilde A^k) &= \mu_1^k + \mu_2^k + \cdots + \mu_m^k  \\
                            & = (d-1)^k + \mu_2^k + \cdots + \mu_m^k\\
                            & = (d-1)^k + O\left( |\mu_2|^k + \cdots + |\mu_m|^k\right) \\
                            & = (d-1)^k + O\left( m \cdot \max_{i\ge 2}|\mu_i|^k  \right).
\end{align*}
By the above spectral gap bound \eqref{eq:spec-gap} on the $\mu_i$, we get that, with probability tending to $1$ as $n\to\infty$, 
\begin{align*}
  N_{tr}(k) = (d-1)^k + O\left( m \left( (d-1)-\delta\right)^k\right).
\end{align*}
Recall that $m=nd$, so when $k\succ \log n$, the above gives
\begin{align}
  N_{tr}(k)= (1+o(1)) \cdot (d-1)^k \label{eq:ntr}
\end{align}
with probability tending to $1$ as $n\to\infty$.

Now we can express $N_{tr}(k)$ in terms of $N_{prim}$. Each loop counted by $N_{tr}$ has a period $r | k$, and there are $r$ distinct choices of starting edge. So we have:
\begin{align*}
  N_{tr}(k)=\sum_{r|k}r \cdot N_{prim}(r).  
\end{align*}
We isolate $N_{prim}(\Gamma,k)$ and use the immediate universal inequality $N_{prim}(\Gamma,r) \le n d^r$ holding for any $\Gamma,r,n$, together with \eqref{eq:ntr} to get, with probability tending to $1$ as $n\to\infty$, 
\begin{align*}
  k \cdot N_{prim}(k) &= N_{tr}(k)-\sum_{r|k, r\ne k} r\cdot N_{prim}(r) \\
                             &= N_{tr}(k)-O\left(\sum_{r|k, r\ne k} n \cdot d^r \right)\\
                           & = (1+o(1)) \cdot (d-1)^k- O(k \cdot n \cdot d^{k/2}) \\
                           &=(1+o(1)) \cdot (d-1)^k - o\left((d-1)^k\right),
\end{align*}
where we have used that $k\succ \log n$ to get the last line. 
The desired result follows.

\end{proof}

\begin{rmk}
  When $d\ge 5$, one can accurately bound the expectation of $N_{prim}(\Gamma,k)$ for $k\prec \log n$ using more quantitative information about the eigenvalues of random regular graphs (namely \cite[Theorem 3.1]{Friedman}).  But these methods do not seem to work for $d=3,4$.    
\end{rmk}

\section{Simple vs primitive loops}
\label{sec:simple-prim}

In this section we prove the main theorems by combining the results that we've proved about counting simple and primitive loops in the previous two sections.  

\subsection{Low length regime}
\label{sec:low-length-regime}

\begin{proof}[Proof of Theorem \ref{thm:graph-short}] 
  We will prove the theorem when the order of growth $k$ lies in one of three specific ranges:
  (1) $k$ constant, (2) $1\prec k\prec n^{1/4}$, (3) $\log n \prec k(n) \prec \sqrt n$.  In general, $k(n)$ need not stay in any of these three ranges; however the general case reduces to these.  In fact, to show that the desired limit of probabilities is $1$, it suffices to show that any subsequence $\{n_i\}$  has a further subsequence $\{n_{i_j}\}$ along which with the limit is $1$.  Given $\{n_i\}$, we can always find $\{n_{i_j}\}$ along which the growth of $k$ falls into one of the three cases.  Hence, by the below, the limit of the probabilities along $\{n_{i_j}\}$ will be $1$.

  \medskip
  \noindent \underline{Case 1:} $k$ is constant.
  \medskip

  Our proof uses (i)  $N_{simp}\le N_{prim}$, (ii) the integrality of $N_{simp},N_{prim}$, and (iii) our results on expectation of $N_{simp}, N_{prim}$, in particular that these converge to the same \emph{constant} value in this regime.  
  
  Since $N_{simp}\le N_{prim}$, and both are valued in non-negative integers, we have
  \begin{align*}
    \E_n[N_{prim}(k)] - \E_n[N_{simp}(k)] \ge \P_n[N_{prim}(k) \ne N_{simp}(k)] \cdot 1 . 
  \end{align*}
  By Proposition \ref{prop:e-simp-short} and Proposition \ref{prop:e-prim-short}, the quantities $\E[N_{prim}]$ and $\E[N_{simp}]$ both converge to the constant $(d-1)^k/k$, and hence the left hand side of the above tends to $0$ as $n\to \infty$.  It follows that
  $$\lim_{n\to\infty} \P_n[N_{simp}(k)\ne N_{prim}(k)] = 0.$$
So for any $\epsilon>0$, we get that
\begin{align*}
  \P_n\left[ N_{simp}(k)\ge (1-\epsilon) N_{prim}(k)\right] \ge \P_n[N_{simp}(k)= N_{prim}(k)] \to 1,
\end{align*}
as $n\to\infty$, as desired.  

\medskip
\noindent \underline{Case 2:} $1 \prec k(n) \prec n^{1/4}$
\medskip

In previous sections we have gained control over the expectation of $N_{simp}$ and $N_{prim}$ in this regime, showing that they have the same asymptotic behavior.  And of course we also have $N_{simp}\le N_{prim}$.  However, these facts alone are not enough to deduce the desired result.  For instance, we need to rule of the situation in which with probability $1/2$, $N_{simp}=1$ and $N_{prim}=2$, and with probability $1/2$, both $N_{simp}$ and $N_{prim}$ are around $2(d-1)^k/k$.  Note that in this case, $N_{simp},N_{prim} \sim (d-1)^k/k$ as $k\to\infty$, but there is a $1/2$ chance that the ratio is equal to $2$.  This type of situation will be ruled out by our control of the a.a.s.\@ behavior of $N_{simp}$ (which was proved by bounding the second moment $N_{simp}$).  

We begin by defining two random variables
\[
 X :=\frac{N_{simp}(\Gamma,k)}{(d-1)^k/k}, \quad Y:=\frac{N_{prim}(\Gamma,k)}{(d-1)^k/k}.
\]
Note that $X \le Y$ everywhere. Moreover, by our assumptions on $k(n)$, we have that $1 \prec k \prec n^{1/4}$. So $\E(X) \sim 1$ and $\E(Y) \sim 1$ by Propositions \ref{prop:e-simp-short} and \ref{prop:e-prim-short}. Thus, $\E(Y) - \E(X) = o(1)$. In particular, for any $\epsilon > 0$, for $n$ large enough
\[
 \E(Y) - \epsilon \le \E(X) \le \E(Y). 
\]

So by Lemma \ref{lem:e-comparison}, we get that 
\begin{align}
 \P_n( X \ge Y - \sqrt \epsilon) \ge 1- \sqrt \epsilon, \label{eq:XvY}
\end{align}
for $n$ large enough.

The above means that $X,Y$ are additively close with high probability, but the desired statement is about multiplicative closeness (which need not follow from additive closeness if both $X,Y$ are small).   We will achieve this by bounding $X$ (and hence $Y$) from below almost surely.  

By Proposition \ref{prop:aas-simple-short}, we have that $Y\ge X>1/2$ with probability at least $1-\epsilon$ for all $n$ large enough.   %
Hence
$$Y-\sqrt{\epsilon}>(1-2\sqrt\epsilon )Y$$
with probability at least $1-\epsilon$ for large $n$.

Combining this with \eqref{eq:XvY} gives that 
\begin{align*}
  \P_n[X\ge (1-2\sqrt\epsilon) Y ] \ge \P_n[X\ge Y-\sqrt\epsilon ] -\epsilon \ge 1-\sqrt\epsilon -\epsilon,
\end{align*}
when $n$ is sufficiently large.  This implies the desired result.

\bigskip\noindent\underline{Case 3:} $\log n \prec k(n) \prec \sqrt n$
\medskip

The proof in this case comes directly from our control of the a.a.s.\@ behavior of both $N_{simp},N_{prim}$ in this regime.  

Applying Proposition \ref{prop:aas-simple-short} and Proposition \ref{prop:aas-prim-long} gives that, for any $\epsilon>0$, with probability approaching $1$ as $n\to\infty$, we have both
\begin{align*}
  N_{simp} \ge (1-\epsilon) (d-1)^k/k,\\
  N_{prim} \le (1+\epsilon) (d-1)^k/k. 
\end{align*}

It follows that
$$N_{simp} \ge \frac{1-\epsilon}{1+\epsilon} \cdot N_{prim},$$
with probability approaching $1$ as $n\to\infty$.  
This implies the desired result.

\medskip

\end{proof}

The below is a version with additive error of the basic probability fact that if one random variable dominates another and they have the same expectation, then they are equal almost everywhere.  
\begin{lemma}
  \label{lem:e-comparison}
  Let $X,Y$ be random variables (on the same probability space) with $X\le Y$ everywhere.  Suppose that for some $\epsilon>0$,
  $$\E[X]\ge \E[Y]-\epsilon.$$
  Then
  $$\P[X\ge Y-\sqrt{\epsilon}]\ge 1-\sqrt{\epsilon}.$$
\end{lemma}

\begin{proof}
Using that $X\le Y$, we have 
  \begin{align*}
    \E[X] %
    \le \E[Y]  -\sqrt{\epsilon}\cdot \P\big[X< Y-\sqrt{\epsilon}\big].
  \end{align*}
Combining this with $\E[Y]-\epsilon \le \E[X]$, we get
\begin{align*}
  \E[Y]-\epsilon \le \E[Y]-\sqrt{\epsilon}\cdot \P\big[X< Y-\sqrt{\epsilon}\big],
\end{align*}
and hence $\P[X<Y-\sqrt\epsilon] < \sqrt\epsilon$.
\end{proof}

\subsection{High length regime}
\label{sec:high-length-regime}

\begin{proof}[Proof of Theorem \ref{thm:graph-long}]

  In this regime the only input we need is the expected count of simple loops and the a.a.s.\@ count of primitive loops.
  
  Let
  \begin{align*}
    X= \frac{N_{simp}(\Gamma,k)}{(d-1)^k/k}, \quad  Y= \frac{N_{prim}(\Gamma,k)}{(d-1)^k/k}.
  \end{align*}

 For the expected simple count, by Proposition \ref{prop:e-simp-long}, 
  \begin{align*}
    \lim_{n\to\infty} \E_n[X] =0. 
  \end{align*}

  We now can apply the first moment method; by Markov's inequality, we have
  \begin{align}
    \label{eq:simp-markov} 
    \P_n[X\ge \epsilon/2] \le \frac{\E_n[X]}{\epsilon/2} \to 0,
  \end{align}
  as $n\to\infty$.

  For the a.a.s.\@ primitive count, by Proposition \ref{prop:aas-prim-long}, we have
  \begin{align}
    \label{eq:prim-aas} 
    \lim_{n\to\infty} \P_n\left[  1-\epsilon < Y < 1+\epsilon \right] = 1. 
  \end{align}

Then
\begin{align*}
  \P_n\left[ N_{simp}(k) \ge \epsilon \cdot N_{prim}(k) \right] =\P_n\left[ X \ge \epsilon \cdot Y \right] \le \P_n[X \ge \epsilon/2] + \P_n[Y\le 1/2],
\end{align*}
and the above goes to $0$ as $n\to\infty$, by \eqref{eq:simp-markov} and \eqref{eq:prim-aas}. 
\end{proof}

 \bibliographystyle{alpha}
  \bibliography{sources}

\end{document}